\documentclass[12pt]{article}

\usepackage{amssymb,amsmath,amsthm,ucs}
  \usepackage[backref=page]{hyperref}

\newcommand{\ovl}{\overline}
\newcommand{\vp}{\varepsilon}
\newcommand{\cl}[1]{{\mathcal{#1}}}

\newcommand{\mxg}{M\rtimes_\alpha G}
\newcommand{\axg}{A\rtimes_\alpha G}
\newcommand{\fmxg}{M\rtimes_\alpha G_F}
\newcommand{\mxh}{M\rtimes_\alpha H}

\newcommand{\ggnmn}{\mathcal{GN}(M\subseteq N)}
\newcommand{\gnmn}{\mathcal{GN}_{\mathcal{Z}(M)}(M\subseteq N)}
\newcommand{\gncp}{\mathcal{GN}_{\mathcal{Z}(M)}(M\subseteq M\rtimes_\alpha G)}
\newcommand{\Ad}{\mathrm{Ad}\,}

\newcommand{\sett}[1]{\left\{#1\right\}}

\newcommand{\Hil}{\mathcal{H}}

\newcommand{\Z}{\mathcal{Z}}

\numberwithin{equation}{section}

\theoremstyle{plain}
\newtheorem{lem}{Lemma}[section]

\newtheorem{thm}[lem]{Theorem}
\newtheorem{cor}[lem]{Corollary}

\theoremstyle{definition}
\newtheorem{defn}[lem]{Definition}

\theoremstyle{remark}
\newtheorem{rem}[lem]{Remark}

\begin{document}

\null\vspace{.5in}
\setcounter{page}{1}
\begin{LARGE}
\begin{center}
INTERMEDIATE SUBALGEBRAS AND BIMODULES  FOR GENERAL CROSSED PRODUCTS OF VON NEUMANN ALGEBRAS
\end{center}
\end{LARGE}
\bigskip

\begin{center}
\begin{tabular}{cc}
{\large Jan Cameron}$^{(*)}$&{ \large Roger R. Smith}$^{(**)}$\\
&\\
Department of Mathematics & Department of Mathematics\\
Vassar College& Texas A\&M University\\
Poughkeepsie, NY 12604 & College Station, TX 77843\\
&\\
jacameron@vassar.edu & rsmith@math.tamu.edu

\end{tabular}
\end{center}

\vfill

\begin{abstract}
Let $G$ be a discrete group acting on a von Neumann algebra $M$ by properly outer $*$-automorphisms. In this paper we study
the containment  $M \subseteq M\rtimes_\alpha G$ of $M$ inside the crossed product. We characterize the intermediate von Neumann algebras, extending earlier work of other authors in the factor case. We also determine the $M$-bimodules that are closed in the Bures topology and which coincide with the $w^*$-closed ones under a mild hypothesis on $G$. We use these results to obtain a general version of Mercer's theorem concerning the extension of certain isometric $w^*$-continuous  maps on $M$-bimodules to $*$-automorphisms of the containing von Neumann algebras.

\end{abstract}

\vfill


\vfill


\vfill

\noindent ($\ast$) JC was partially supported by Simons Collaboration Grant for Mathematicians \#319001

\noindent ($\ast \ast$) RS was partially supported by NSF grant DMS-1101403.
\newpage

\section{Introduction}\label{sec1}

In this paper we address several problems that arise in the theory of crossed products. Before discussing our results, we begin by presenting some background.
 
Let a discrete group $G$ act on a von Neumann algebra $M$ by properly outer $*$-automorphisms. When $M$ is a factor, the resulting crossed product $\mxg$ is also a factor, and each subgroup $H$ of $G$ gives rise to a natural intermediate  factor $M\subseteq \mxh\subseteq\mxg$. In \cite{Cho}, Choda considered the question of whether there is a Galois theory in this situation: do the crossed products by subgroups of $G$ account for all intermediate factors? He was able to give an affirmative answer under an additional hypothesis that each intermediate factor $N$ is the range of a faithful normal conditional expectation $E:\mxg\to N$, a condition that is automatically satisfied when $M$ is a type II$_1$ factor. For factors $M$ with separable preduals, the existence of faithful normal conditional expectations onto all intermediate factors was established in \cite{ILP}, providing a Galois theory under the separable predual assumption. This separability requirement was removed in \cite{CaSm}, so that Choda's correspondence between intermediate subfactors and subgroups holds for all factors.

At the other end of the spectrum lie the abelian von Neumann algebras, and the crossed products by discrete groups give particular examples of Cartan inclusions $A\subseteq\axg$. In \cite{MSS}, Muhly, Saito and Solel considered general Cartan inclusions $A\subseteq M$ where $M$ is a von Neumann algebra and $A$ is a Cartan masa in $M$. Their interest lay in the triangular subalgebras of $M$ for which a detailed analysis of the $w^*$-closed $A$-bimodules was an essential ingredient.
 Writing the  masa $A$ as
$L^\infty(X,\mu)$ for a measure space, the theory of Feldman and Moore \cite{FM1,FM2} describes $M$ in terms of   a measurable equivalence relation $R\subseteq X\times X$ and 2-cocycle on $R$ arising from the action of the normalizing unitaries $\mathcal{N}(A \subseteq M)$ on $A$. The Spectral Theorem for Bimodules \cite[Theorem 2.5]{MSS} asserted that the $w^*$-closed $A$-bimodules could be
characterized in terms of measurable subsets of the equivalence relation, and this was the main technical device for their subsequent discussion of various types of
$w^*$-closed subalgebras and isomorphisms between them. Unfortunately, as noted in \cite{Aoi}, there was a gap in the proof and so the question of whether this theorem is correct remains open,
although it has been proved in \cite{CPZ} for a (possibly) smaller class of $A$-bimodules that are closed in a different topology. This is a topology introduced by Bures \cite{Bur}, and which we term the $B$-topology. Its definition, which we give in Section \ref{bimodsect},
requires a containment $P\subseteq Q$ of von Neumann algebras with a faithful normal conditional expectation $E:Q\to P$, so applies to Cartan inclusions $A\subseteq M$ as well as those of the form $M\subseteq \mxg$. It plays a vital role in the theory of crossed products because the Fourier series $\sum_{g\in G}x_gg$ of an element $x\in \mxg$ need not converge in the standard von Neumann algebra topologies, but does converge to $x$ in the $B$-topology \cite{Mer0}. For the original formulation of the Spectral Theorem for Bimodules in terms of $w^*$-closed bimodules, only one case has been verified: a Cartan inclusion $A\subseteq M$ where $M$ is an amenable von Neumann algebra \cite{Ful}.

Our objective is to develop a general theory of bimodules and subalgebras of crossed products that includes the two types discussed above. If we choose a central projection $z_g\in \Z(M)$ then a simple example of an $M$-bimodule in $\mxg$ is $Mz_gg$. These form the building blocks of the theory since we show that every $B$-closed $M$-bimodule is the $B$-closed span of these special $M$-bimodules. When we require the group $G$ to have 
 the approximation property ({\emph{AP}}) of Haagerup and Kraus \cite{HK} (a large class of groups that includes the weakly amenable ones of Cowling and Haagerup \cite{CH}) we find that the $w^*$-closed and $B$-closed $M$-bimodules coincide. As a consequence, the Spectral Theorem for Bimodules is valid for any Cartan inclusion $A\subseteq \axg$ where $A$ is abelian and $G$ has the  {\emph{AP}}. In particular, this gives the first classes of
Cartan inclusions beyond the realm of amenable von Neumann algebras for which this theorem has been verified. 

There are some natural candidates for intermediate von Neumann algebras $M\subseteq N\subseteq \mxg$. For a set of central projections $z_g$ indexed by $G$, form the $w^*$-closed bimodule $N$ that is the $w^*$-closed span of $\{Mz_gg:g\in G\}$. Then there are some algebraic relations among the $z_g$'s to ensure that $N$ is a von Neumann algebra, and we prove that every intermediate von Neumann algebra has this form for a suitable choice of $z_g$'s. In \cite{Tak0} Takesaki related the existence of normal faithful conditional expectations to properties of the modular automorphism groups of certain faithful normal semifinite weights, and we use his theory to show that each intermediate $N$ is the range of such a conditional expectation. As a consequence, we show that any such $N$ is $B$-closed, and deduce from this the concrete form of the conditional expectation: on generators $mg$ for $\mxg$, it is $mg\mapsto mz_gg$.

A theorem of Mercer \cite{Mer} asserted, in the case of a Cartan masa inclusion $A\subseteq M$, that an isometric and surjective $A$-bimodule map on a $w^*$-closed
$A$-bimodule $X$ that generates $M$ has an extension to a $*$-automorphism of $M$, thus highlighting the essential rigidity of such maps. However, the argument was
flawed, and this result was established subsequently for $w^*$-continuous maps in \cite{CPZ}, a hypothesis that was implicit in the original formulation of \cite{Mer}. We apply  our results on the bimodules and intermediate von Neumann algebras to
 prove a general version of Mercer's extension theorem in Theorem \ref{liftX0} for crossed product inclusions $M\subseteq \mxg$, extending the Cartan result of \cite{CPZ} when the
Cartan masa arises from $A\subseteq A\rtimes_\alpha G$ for an abelian von Neumann algebra $A$. For examples of such extension theorems in the context of {\emph{CSL}}-algebras, see \cite{DavP}. We note that in the special case when $M$ is a factor we have already established these results in \cite{CaSm} using methods that were particular to factors.

The paper is organized as follows. Section \ref{bimodsect} has a brief discussion of the $B$-topology and we characterize the $B$-closed bimodules in terms of systems of
central projections indexed by the group $G$ (Theorem \ref{genmod}). Our classes of Cartan inclusions for which the Spectral Theorem for Bimodules is valid are presented in
Theorem \ref{STB}. In Section \ref{intermed} we characterize the intermediate von Neumann algebras for the inclusion $M\subseteq \mxg$ (Theorem \ref{Struct}). These all have the
form $N=\ovl{\mathrm{span}}^{w^*}\{Mz_gg:g\in G\}$ where the $z_g$'s are central projections satisfying some algebraic relations that ensure that $N$ is a unital
$*$-algebra.

The next two sections contain preliminary results that lead up to the proof of our main extension theorem for crossed products in Theorem \ref{liftX0}. Section \ref{groupoidsect} deals with certain central
groupoid normalizers, and our approach is based on results from Connes's noncommutative Rokhlin theorem \cite{Con}. These are used to establish Lemma \ref{ideal} which shows that a
nonzero ideal in a C$^*$-algebra containing $M$ will have  nonzero intersection with $M$ under appropriate hypotheses. This is the main result of the section, but we
also include automatic $B$-continuity of certain conditional expectations and $*$-automorphisms for later use. We also examine containments of the form $M\subseteq A$ where $A$ is a C$^*$-algebra generated by central groupoid normalizers of $M$. Here we show that if a conditional expectation onto the smaller algebra exists then it must be unique, in contrast to what may happen more generally \cite{AS}. Section \ref{normingsect} establishes that $M$ always norms $\mxg$
(as defined in \cite{PSS}), and this is used subsequently to prove that isometric $M$-bimodule maps are automatically completely isometric.

In Section \ref{Mercersect} we establish a version of Mercer's  extension theorem for general crossed products. Starting from a $w^*$-closed $M$-bimodule $X$, we drop to a submodule $X_0$
generated by the central groupoid normalizers in $X$ and we extend an $M$-bimodule map $\phi: X\to X$ to a $*$-automorphism of C$^*(X_0)$, which is the C$^*$-envelope of
$X_0$. Automorphisms of C$^*$-algebras do not necessarily extend to automorphisms of their containing von Neumann algebras so further argument, in which the $B$-topology
plays a crucial role, is needed to finally obtain an automorphism of $W^*(X)$ which extends the map $\phi$.

We close by noting that our approach to the problems discussed in this paper has been inspired by the ideas and techniques developed in \cite{CPZ,Mer0,MSS} and we gratefully
record our indebtedness to these authors. We also thank Stuart White for his comments on an earlier draft of the paper.

\section{Structure of bimodules}\label{bimodsect}

Let $A\subseteq N$ be an inclusion of von Neumann algebras with a  faithful normal conditional expectation $E:N\to A$, the example of interest for us being $M\subseteq
\mxg$ with conditional expectation $E_M$.  Here $G$ is a discrete group acting on $M$ by properly outer
$*$-automorphisms $\{\alpha_g:g\in G\}$, a slight abuse of terminology
since $\alpha_e=\,$id. We remind the reader that a $*$-automorphism $\alpha$ of a von Neumann algebra $M$ is inner if there exists a unitary $u\in M$ so that $\alpha=\Ad u$, while $\alpha$ is properly outer if there does not exist a nonzero $\alpha$-invariant central projection $z\in M$ so that $\alpha|_{Mz}$ is inner. Bures \cite{Bur} introduced a topology on $N$, which we term the $B$-topology, defined by the seminorms
\begin{equation}
\omega(E(x^*x))^{1/2},\ \ \ x\in N,
\end{equation}
where $\omega$ ranges over the normal states of $A$. Convergence of a net $\{x_\lambda\}_{\lambda\in \Lambda}$ from $N$ to $x\in N$ means that
\begin{equation}
\lim_\lambda \omega(E((x-x_\lambda)^*(x-x_\lambda)))=0,\ \ \ \omega\in A_*,
\end{equation}
from which it follows that the $B$-topology is weaker than the $\sigma$-strong topology. Throughout we use the term $w^*$-topology to denote the $\sigma$-weak topology on a von Neumann algebra. 
We refer to \cite[Section 3]{CaSm} for a discussion of the $B$-topology, but we include the following lemma which is implicit in the proof of \cite[Lemma 3.2]{CaSm}.
\begin{lem}\label{Bconv}
Let $A\subseteq N$ be an inclusion of von Neumann algebras with a faithful normal conditional expectation $E:N\to A$. A uniformly bounded  net from $N$
is $B$-convergent convergent if and only if it is $\sigma$-strongly convergent. In particular, such a $B$-convergent net is also convergent in the $w^*$-topology.
\end{lem}

The crossed product $mxg$ is generated by a faithful normal
representation $\pi$ of $M$ and unitary operators $\{u_g:g\in G\}$ so that
$u_g\pi(m)u_g^*=\pi(\alpha_g(m))$ for $m\in M$ and $g\in G$ \cite{vD}. For ease of
notation we suppress $\pi$ and $u$, and write the generators of $\mxg$ as
$\{mg:m\in M,\ g\in G\}$ subject to the relations $gm=\alpha_g(m)g$.
Each element $x\in \mxg$ can be represented by a Fourier series $\sum_{g\in G}x_gg$,  and there is a unique faithful normal conditional expectation $E_M:\mxg
\to M$ given by $E_M(x)=x_e$. The remaining Fourier coefficients are then
defined by $x_g=E_M(xg^{-1})$ for $x\in \mxg$ and $g\in G$. As shown by Mercer \cite{Mer0}, this series converges to
$x$ in the $B$-topology, while it may fail to converge in the more standard von Neumann algebra topologies. Consequently the $B$-topology will play a central role in
this paper.

We now consider the structure of $M$-bimodules inside $\mxg$, generalizing the results of \cite{CaSm} in the factor case.

\begin{lem}\label{Mzg}
Let a discrete group $G$ act on a von Neumann algebra $M$ by properly outer $*$-automorphisms. Let $X\subseteq \mxg$ be a $w^*$-closed $M$-bimodule. Then there exist
central projections $\{z_g\in \mathcal{Z}(M):g\in G\}$ with the following properties:
\begin{itemize}
\item[(i)]
$Mz_gg\subseteq X$,
\item[(ii)] if $x\in X$ has Fourier series $\sum_{g\in G}x_gg$, then $x_g\in Mz_g$ for $g\in G$ and $x_gg\in X$.
\end{itemize}
\end{lem}

\begin{proof}
For $g\in G$, let $I_g:=\{m\in M:mg\in X\}$. Then each $I_g$ is a $w^*$-closed ideal in $M$ so has the form $Mz_g$ for a central projection $z_g\in \mathcal{Z}(M)$. Then
$Mz_gg\subseteq X$, verifying (i).

Now let $J_g$ be the algebraic ideal in $M$ of elements that appear as the $g$-coefficient in the Fourier series of some element of $X$. Then $I_g\subseteq J_g$, so (ii)
will be verified by showing that $J_g\subseteq I_g$. Let $g_0\in G$ be fixed but arbitrary and consider a projection $p\in J_{g_0}$ with central support $z\geq p$. There
is at least one such projection since $0\in J_{g_0}$. By \cite[p. 61]{SS2} there exist partial isometries $v_i\in M$ such that $\sum_i v_i^*pv_i=z$. If $x=\sum_{g\in
G}x_gg\in X$ with $x_{g_0}=p$, then define $y=\sum_{g\in G}y_gg\in X$ by $y=\sum_iv_i^*x\alpha_{{g_0}^{-1}}(v_i)$, the sum converging in the $w^*$-topology. We note that
the $g_0$-coefficent is $y_{g_0}=\sum_i v_i^*pv_i=z$. Thus $z\in J_g$. By \cite[Theorem 3.3]{ChSin}, there is a net of sets of elements $\beta=(m_j)_{j\in J}$, $m_j\in M$, such that
\begin{itemize}
\item[(i)] $\sum_{j\in J} m_j^*m_j =1$;
\item[(ii)] for each completely bounded map $\phi :M\to M$ and each $x\in M$, $\lim_{\beta}\sum_{j\in J} \phi(xm_j^*)m_j$ exists in the $w^*$-topology and defines a
completely bounded linear right $M$-module map denoted $\rho(\phi)$;
\item[(iii)] $\phi \mapsto \rho(\phi)$ defines a completely contractive map $\rho :L_{cb}(M,M)\to L_{cb}(M,M)_M$ where these are respectively the spaces of completely
bounded maps on $M$ and of completely bounded right $M$-module maps on $M$.
\end{itemize}

It was shown in \cite[Lemma 4.2]{CaSm} that if $\alpha $ is a properly outer $*$-automorphism of $M$, $a\in M$, and $\phi(x):=\alpha(x)a$ for $x\in M$, then $\rho(\phi)=0$.

Now form the element
\begin{equation}
y^\beta =\sum_{j\in J} m_j^*y\alpha_{{g_0}^{-1}}(m_j)\in X,
\end{equation}
and note that $\|y^\beta\|\leq \|y\|$.
For each $\beta$, the $g_0$-coefficient of $y^\beta$ is
\begin{equation}
\sum_{j\in J} m_j^*zm_j=\sum_{j\in J} zm_j^*m_j=z,
\end{equation}
using the centrality of $z$. Drop to a $w^*$-convergent subnet of $\{y^\beta\}$ and follow the proof of \cite[Theorem 4.3]{CaSm} to conclude that the $w^*$-limit of this
subnet is $zg_0$ since the other terms of the Fourier series tend to 0 using the previous remarks on the map $\rho$. Thus $zg_0\in X$ so $z\in I_{g_0}$. Thus $p\in
I_{g_0}$ since $p=pz$.

Consider now a general element $m\in J_{g_0}$. Then $m^*m\in J_{g_0}$, and let $p_n$ be the spectral projection of $m^*m$ corresponding to the interval
$[n^{-1},\infty)$, $n\geq 1$. Choose $N$ so that $p_n\ne 0$ for $n\geq N$. Since $p_nm^*mp_n$ is invertible in $p_nMp_n$ for $n\geq N$, we may multiply by its inverse to obtain that $p_n$ is a projection in $J_{g_0}$  so lies in
$I_{g_0}$ from above. Thus $mp_n\in I_{g_0}$, and we conclude that $m\in I_{g_0}$ since $\lim_{n\to \infty}\| m-mp_n\|=0$. This establishes that $J_{g_0}\subseteq I_{g_0}$,
completing the proof.
\end{proof}

\begin{rem}\label{outer}
In the proof of Theorem \ref{liftX0} we will need the following fact. Let $\beta=(m_j)$ be the operators in the proof of Lemma \ref{Mzg} and let $\psi$ be a properly outer $*$-automorphism of a von Neumann algebra $M$. Let $a\in M$ be fixed. Then
\begin{equation}\label{outer.1}
w^*\text{-}\lim_\beta \sum_j m_j^*a\psi(m_j)=0.
\end{equation}
This is an immediate consequence of \cite[Lemma 4.2]{CaSm} applied to the completely bounded map $\phi:x\mapsto \psi(x)a$ since the left hand side of \eqref{outer.1} is the limit defining $(\rho\phi)(1)$.\hfill$\square$
\end{rem}

We can now give a description of the $B$-closed $M$-bimodules in $\mxg$ as well as the $w^*$-closed $M$-bimodules when $G$ has the Approximation Property ({\em{AP}}) of
Haagerup and Kraus \cite{HK}. This result parallels and extends \cite[Theorem 4.4]{CaSm}, using Lemma \ref{Mzg} in place of \cite[Theorem 4.3]{CaSm}. The argument is
essentially the same and so we omit the details.

\begin{thm}\label{genmod}
Let a discrete group $G$ act on a von Neumann algebra $M$ by properly outer $*$-automorphisms.
\begin{itemize}
\item[(i)]
There is a bijective correspondence between sets $S=\{z_g:g\in G\}$ of central projections in $M$ and $B$-closed $M$-bimodules in $\mxg$ given by \begin{equation}
S\mapsto X_S:=\ovl{\mathrm{span}}^B\{xz_gg:x\in M,\ g\in G,\ z_g\in S\}.
\end{equation}
\item[(ii)]
If $G$ has the ${\mathit{AP}}$, then the collections of $B$-closed and $w^*$-closed $M$-bimodules coincide.
\end{itemize}
\end{thm}

We note here some connections between the results in this section, and those of \cite{CPZ} in the analogous setting of von Neumann algebras with a Cartan masa.   The two settings overlap in  crossed products of the form $L^\infty(X,\mu) \rtimes_\alpha  G,$ where $\alpha$ is induced by a free, ergodic action of $G$ on $X.$  In this special case, some of the results in \cite{CPZ} -- principally, Lemma 2.3.1 and Theorem 2.5.1 --  could in fact be used to give an alternate proof of part (i) of Theorem \ref{genmod}.  However, the techniques in \cite{CPZ} are largely specific to masas, so do not apply to more general crossed products.  On the other hand, the corollary of part (i) of Theorem \ref{genmod} stated below is an analogue for crossed products of the equivalence of conditions (c) and (d) in Theorem 2.5.1 of \cite{CPZ}. Before giving the full statement, we require some terminology.

\begin{defn}\label{groupoiddef}
Let $M\subseteq N$ be an inclusion of von Neumann algebras. A partial
isometry $v\in N$ is said to be a groupoid normalizer of $M$ if
$vMv^*,\,v^*Mv\subseteq M$, and the set of such elements is denoted by $\ggnmn$.
If, in addition, the projections $vv^*$ and $v^*v$ are central in $M$, then
we refer to $v$ as a central groupoid normalizer. We denote the set of
central groupoid normalizers by $\gnmn$, and we note that this set is closed under the operations of taking adjoints and products.\hfill$\square$
\end{defn}
 
 \begin{cor} \label{groupoid}  Let a discrete group $G$ act on a von Neumann algebra $M$ by properly outer $*$-automorphisms. Then any $B$-closed $M$-bimodule $X \subseteq \mxg$ has the form
\begin{equation}
X = \ovl{\mathrm{span}}^B\{X \cap \mathcal{GN}_{\mathcal{Z}(M)}(M \subseteq \mxg)\}.
\end{equation}
\end{cor}

Combining this result with part (ii) of Theorem \ref{genmod}, we obtain the following characterization of $w^*$-closed bimodules in crossed products by groups with the $\mathit{AP}$.  

\begin{cor} \label{STBcor} Let a discrete group $G$ act on a von Neumann algebra $M$ by properly outer $*$-automorphisms. If $G$ has the $\mathit{AP},$ then any $w^*$-closed $M$-bimodule $X \subseteq \mxg$ has the form 
\begin{equation}
X = \ovl{\mathrm{span}}^{w^*}\{X \cap \mathcal{GN}_{\mathcal{Z}(M)}(M \subseteq \mxg)\}.
\end{equation}
\end{cor}
We conclude this section with some observations  on the Spectral Theorem for Bimodules of Muhly, Saito, and Solel \cite{MSS}.
If $M$ is a von Neumann algebra with separable predual, and a Cartan masa $A \subseteq M,$ then by the Feldman-Moore theory \cite{FM2} there exist a standard Borel space $(X,\mu)$, a Borel equivalence relation $R \subseteq X \times X$ with the property that each equivalence class is countable,  and a $2$-cocycle $\sigma$ on $R$ such that $M$ is isomorphic to the von Neumann algebra $M(R,\sigma)$ generated by $R$ and $\sigma$ (see \cite{FM2} for details on this construction).  The von Neumann algebra $M(R,\sigma)$ is a subalgebra of $B(L^2(R,\nu))$, where $\nu$  is the right counting measure on $R$ associated to $\mu$.   Denote by $\xi$ the characteristic function of the diagonal subset $\Delta_R =\sett{(x,x):x \in X}$ of $R$. Then, the map $ a \mapsto a \xi $ defines an embedding of $M(R,\sigma)$ into $L^2(R,\sigma)$, and the isomorphism of $M$ with $M(R,\sigma)$ identifies $A$ with the subalgebra $A(R,\sigma)$ of operators $a \in M(R,\sigma)$ such that $a\xi$ is supported on $\Delta_R.$  More generally, given a  measurable subset $S \subseteq R,$  it is straightforward to check that the set
\[  Y_S = \sett{ a \in M(R,\sigma): a\xi \text{ is supported on } S} \subseteq M(R, \sigma)\]
is a $w^*$-closed $A$-bimodule in $M(R,\sigma)$.  In \cite{MSS}, which studied structural properties of $w^*$-closed subalgebras of von Neumann algebras $M(R,\sigma)$ with Cartan masas, it was asserted, with the above notations, that every $w^*$-closed $A$-bimodule in $M(R,\sigma)$ has the form $Y_S$ for some $S \subseteq R.$  This statement was termed the Spectral Theorem for Bimodules; however, the proof given in \cite{MSS} was flawed (as pointed out by Aoi in \cite{Aoi}), as was a subsequent attempt \cite{MerC} to repair the proof.  Fulman \cite{Ful} gave a proof of the result for the special case in which the containing von Neumann algebra $M(R,\sigma)$ is hyperfinite.  Whether the Spectral Theorem for Bimodules holds in full generality remains an open question.  

It was shown in \cite{CPZ} that, for a von Neumann algebra $M$ with separable predual and a Cartan masa $A$, equality of the collections of $w^*$-closed and $B$-closed $A$-bimodules in $M$ is equivalent to the original statement of the Spectral Theorem for Bimodules in \cite{MSS}.   Thus, in Theorem \ref{genmod} part (ii) above, taking $M$ to be $L^\infty(X,\mu)$, and $\alpha$ to be a free action of a non-amenable group $G$ with the ${\mathit{AP}}$ on $X$ yields the first non-hyperfinite class of von Neumann algebras for which the Spectral Theorem for Bimodules has been verified.   To state this result in the original notation of \cite{MSS}, note that in this case $L^\infty(X,\mu)\rtimes_\alpha G$ is isomorphic to the von Neumann algebra $M(R,1),$ where $R$ is the orbit equivalence relation for the action $\alpha$, and $1$ is the trivial cocycle on $R$.  Denote by $\nu$ the right counting measure on $R$ induced by $\mu.$

 \begin{thm}  \label{STB} Let $G$ be a non-amenable group with the ${\mathit{AP}}$, and $\alpha$ a free action of $G$ on a standard probability space $(X,\mu)$.  Denote by $R$ the orbit equivalence relation for $\alpha,$ write $M = L^\infty(X,\mu) \rtimes_\alpha G \cong M(R,1)$, and let $A$ denote the copy of $L^\infty(X,\mu)$ in $M$.  If $Y \subseteq M$ is a $w^*$-closed $A$-bimodule, then there is a  $\nu$-measurable subset $S$ of $R$ such that 
 \[ Y = Y_S = \sett{a \in M: a \xi \text{ is supported on } S}.\]

 \end{thm}

\section{Intermediate von Neumann algebras}\label{intermed}

Our main objective in this section is to establish Theorem \ref{Struct} which characterizes the von Neumann algebras lying between $M$ and $\mxg$. We will require
some preliminary results that rely on the $B$-topology. The next lemma establishes the $B$-continuity of some conditional expectations and $*$-automorphisms, as well as the relations between them. This will be refined in Corollary \ref{NBcont}, Lemma \ref{Conlem4}, and Corollary \ref{autocont}, but these subsequent results cannot be proved at this stage. All will be used in Section \ref{Mercersect}. 

\begin{lem}\label{Bcont}
Let a discrete group $G$ act on a von Neumann algebra $M$ by properly outer $*$-automorphisms and let $E_M: \mxg \to M$ be the faithful normal conditional expectation.
\begin{itemize}
\item[(i)]
Let $\sigma$ be a $*$-automorphism of $\mxg$ such that $\sigma(M)=M$. Then 
\begin{equation}\label{autocond}
\sigma^{-1}\circ E_M\circ\sigma =E_M
\end{equation}
 and $\sigma$ is $B$-continuous.
\item[(ii)]
Let $N$ be a von Neumann algebra satisfying $M\subseteq N\subseteq \mxg$ and suppose that there exists a faithful normal conditional expectation $E_N:\mxg\to N$. Then
$E_N$ is $B$-continuous.
\end{itemize}
\end{lem}

\begin{proof}
Since $\sigma^{-1}\circ E_M\circ\sigma$ is also a normal conditional expectation onto $M$, equality with $E_M$ follows immediately from uniqueness of the normal conditional
expectation.

Suppose that $\{x_\lambda\}_{\lambda\in\Lambda}$ is a net from $\mxg$ converging to 0 in the $B$-topology, and let $\omega\in M_*$ be a normal state on $M$. Then, using \eqref{autocond} for the second equality,
\begin{equation}
\omega\circ E_M(\sigma(x_\lambda)^*\sigma(x_\lambda))=\omega\circ E_M(\sigma(x_\lambda^* x_\lambda))=\omega\circ \sigma(E_M(x_\lambda^* x_\lambda))
\end{equation}
so the limit over $\lambda\in \Lambda$ is 0 since $\omega\circ \sigma$ is a normal state and $x_\lambda \to 0$ in the $B$-topology. Thus $\sigma(x_\lambda)\to 0$ in the
$B$-topology, so $\sigma$ is $B$-continuous, completing the proof of part (i).

Now suppose that $x_\lambda \to 0$ in the $B$-topology of $\mxg$. Since $E_N$ is a unital completely positive map, \cite[Corollary 2.8]{Choi74} gives
\begin{equation}
E_N(x_\lambda^*)E_N(x_\lambda)\leq E_N(x_\lambda^* x_\lambda),
\end{equation}
so
\begin{equation}
E_M(E_N(x_\lambda)^*E_N(x_\lambda))\leq E_M(x_\lambda^* x_\lambda),
\end{equation}
since $E_M\circ E_N=E_M$. 
Apply an arbitrary normal state $\omega$ on $M$ to this inequality to get
\begin{equation}
0\leq \omega\circ E_M(E_N(x_\lambda)^*E_N(x_\lambda))\leq \omega \circ E_M(x_\lambda^* x_\lambda)
\end{equation}
 and the last term tends to 0 when we take the limit over $\lambda\in\Lambda$. Thus $\lim_{\lambda}E_N(x_\lambda)=0$ in the $B$-topology, showing that $E_N$ is
$B$-continuous.
\end{proof}

We fix a faithful normal semifinite weight $\phi$ of $M$ with associated modular automorphism group $\{\sigma^\phi_t:t\in \mathbb R\}$ \cite[p. 92]{Tak2}. As in \cite{Haa1,Haa2}, the
dual weight $\tilde{\phi}$ on $\mxg$ is $\phi\circ E_M$. Our use of the modular automorphism groups and dual weights below is inspired by the arguments in \cite{CPZ} to obtain parallel results in the Cartan masa context.

We now consider a von Neumann algebra $N$ lying between $M$ and $\mxg$. From Lemma \ref{Mzg}, there exist $w^*$-closed ideals $\{I_g:g\in G\}$ such that $I_gg\subseteq
N$ and if $x\in N$ has Fourier series $\sum_{g\in G}x_gg$ then $x_g\in I_g$ for $g\in G$. Let $N_0$ and $\ovl{N}^B$ be respectively the $w^*$- and $B$-closures of
$\sum_{g\in G}I_gg$. Then $N_0\subseteq N\subseteq \ovl{N}^B$ and $N_0$ is a von Neumann algebra while a similar statement for $\ovl{N}^B$ is not clear since the adjoint
operation is not in general $B$-continuous. Nevertheless, we will show equality of these three spaces.

Recall that a  von Neumann algebra $M$ is said to be countably decomposable if any family of pairwise orthogonal nonzero projections in $M$ is at most countable. The term $\sigma$-finite is also used \cite[Definition II.3.18]{Tak1}.
\begin{thm}\label{exp}
Let a discrete group $G$ act on a von Neumann algebra $M$ by properly outer $*$-automorphisms and let $N$ be a von Neumann algebra lying between $M$ and $\mxg$. Then
$N_0=N=\ovl{N}^B$, and there exists a faithful normal conditional expectation $E_N:\mxg\to N$.
\end{thm}
\begin{proof}
By \cite[Proposition II.3.19]{Tak1}, a projection $p\in M$ is countably decomposable if and only if $pMp$ has a faithful normal state. A simple maximality argument gives a family $\{p_\lambda\}_{\lambda\in\Lambda}$ of countably decomposable projections in $M$ with sum 1, and we specify a  faithful normal semifinite weight $\phi$ on $M$ by first choosing  faithful normal states $\phi_\lambda$ on $p_\lambda M p_\lambda$ for $\lambda\in\Lambda$, and then defining
\begin{equation}
\phi(x)=\sum_{\lambda\in\Lambda}\phi_\lambda (p_\lambda xp_\lambda),\ \ \ x\in M^+.
\end{equation}
Let $\tilde{\phi}=\phi\circ E_M$ be the dual weight on $\mxg$. For $x\in N^+$,
$\tilde{\phi}(p_\lambda xp_\lambda)=\phi_\lambda(p_\lambda E_M(x)p_\lambda)<\infty$, and it follows easily from this that 
$\tilde{\phi}|_N$ is semifinite, a requirement for quoting \cite[Theorem IX.4.2]{Tak2} below, and which dictated our choice of $\phi$.

 From \cite[Theorem X.1.17]{Tak2} $\sigma^{\tilde{\phi}}_t$ and $\sigma^\phi_t$ agree on $M$ and there exist unitaries $\{u_{g,t}\in M:g\in G,\ t\in \mathbb{R}\}$
($(D\phi\circ\alpha_g:D\phi)_t$ in the notation of \cite[Definition VIII.3.4]{Tak2}) so that
\begin{equation}\label{exp.1}
\sigma^{\tilde{\phi}}_t(g)=gu_{g,t}=\alpha_g(u_{g,t})g,\ \ \ t\in \mathbb{R},\ \ \ g\in G.
\end{equation}
Let $\{I_g:g\in G\}$ be the ideals defining $N_0$ and let $\{z_g:g\in G\}$ be the central projections such that $I_g=Mz_g$. Each modular automorphism leaves $\mathcal{Z}(M)$ pointwise fixed, by \cite[Proposition VI.1.23]{Tak2}, so
\begin{equation}\label{exp.2}
\sigma^{\tilde{\phi}}_t(xz_g)=\sigma^\phi_t(xz_g)=\sigma^\phi_t(x)z_g\in I_g.
\ \ \ x\in M,\ \ \ g\in G\ \ \ t\in \mathbb{R},
\end{equation}
Thus
\begin{equation}\label{exp.3}
\sigma^{\tilde{\phi}}_t(xz_gg)=\sigma^\phi_t(x)z_g\alpha_g(u_{g,t})g\in I_gg,\ \ \ x\in M,\ \ \ g\in G,\ \ \ t\in \mathbb{R},
\end{equation}
from \eqref{exp.1}. Thus $\sigma^{\tilde{\phi}}_t(I_gg)\subseteq I_gg$ and equality holds by replacing $t$ by $-t$. It follows that $\sigma^{\tilde{\phi}}_t(N_0)=N_0$ for
$t\in \mathbb{R}$, and so \cite[Theorem IX.4.2]{Tak2} (see also \cite{Tak0}) gives the existence of a
 normal conditional expectation $E_{N_0}:\mxg\to N_0$ satisfying $\tilde{\phi}=\tilde{\phi}\circ E_{N_0}$, and which is thus faithful.  

For $g\in G$,
\begin{equation}\label{exp.4}
gx=\alpha_g(x)g,\ \ \ x\in M,
\end{equation}
so
\begin{equation}\label{exp.5}
E_{N_0}(g)x=\alpha_g(x)E_{N_0}(g)=gxg^{-1}E_{N_0}(g),\ \ \ x\in M,\ \ \ g\in G,
\end{equation}
and it follows that $g^{-1}E_{N_0}(g)\in M'\cap ( \mxg)=\cl{Z}(M)$. Thus there exists $w_g\in \cl{Z}(M)$ so that
\begin{equation}\label{exp.6}
E_{N_0}(g)=gw_g=\alpha_g(w_g)g\in I_gg.
\end{equation}
Consequently $\alpha_g(w_g)z_g=\alpha_g(w_g)$. Since $E_{N_0}(z_gg)=z_gg$, we obtain $z_g\alpha_g(w_g)g=z_gg$ from \eqref{exp.6}, and hence that $\alpha_g(w_g)=z_g$. Thus $E_{N_0}(g)=z_gg$, so
$E_{N_0}(xg)=xz_gg$ for $x\in M$ and $g\in G$. From Lemma \ref{Bcont}, the $B$-continuity of $E_{N_0}$ gives the formula
\begin{equation}\label{exp.7}
E_{N_0}(\sum_{g\in G} x_gg)=\sum_{g\in G} x_gz_gg
\end{equation}
for each $x\in \mxg$ with Fourier series $\sum_{g\in G} x_gg$. It follows that the range of $E_{N_0}$ is $\ovl{N}^B$, proving that $N_0=N=\ovl{N}^B$.
\end{proof}

If a von Neumann algebra $N$ lying between $M$ and $\mxg$ is expressible in terms of a  system of central projections $\{z_g\in \cl{Z}(M):g\in G\}$ by
\begin{equation}
N=\ovl{{\mathrm{span}}}^{w^*}\{Mz_gg:g\in G\},
\end{equation}
then the requirements of being unital, $*$-closed, and multiplicatively closed lead to three respective conditions that must be satisfied:
\begin{align}
z_e&=1,\label{c1a}\\
z_g&=\alpha_g(z_{g^{-1}}),\ \ \ g\in G,\label{c2a}\\
z_{gh}&\geq z_g\alpha_g(z_h),\ \ \ g,h\in G.\label{c3a}
\end{align}
We thank the referee for pointing out that these conditions are equivalent to the formally stronger ones
\begin{align}
z_e&=1,\label{c1}\\
z_g&=\alpha_g(z_{g^{-1}}),\ \ \ g\in G,\label{c2}\\
z_gz_{gh}&= z_g\alpha_g(z_h),\ \ \ g,h\in G.\label{c3}
\end{align}
It is clear that \eqref{c3} implies \eqref{c3a}. For the reverse, we first note  that \eqref{c2a} and \eqref{c3a} imply that
\begin{equation}
\alpha_{g^{-1}}(z_gz_{gh})=\alpha_{g^{-1}}=z_{g^{-1}}\alpha_{g^{-1}}(z_{gh})\leq z_h,
\end{equation}
and so $z_gz_{gh}\leq \alpha_g(z_h)$. Thus $z_gz_{gh}\leq \alpha_g(z_h)$ and so $z_gz_{gh}\leq z_g\alpha_g(z_h)$. Then \eqref{c3} follows from this inequality and \eqref{c3a}.
As we now show, the conditions \eqref{c1}--\eqref{c3} are sufficient for characterizing the intermediate von Neumann algebras.

\begin{thm}\label{Struct}
Let $G$ be a discrete group acting on a von Neumann algebra $M$ by properly outer $*$-automorphisms and let $N$ be a linear space satisfying
$M\subseteq N \subseteq \mxg$. Then $N$ is a von Neumann algebra if and only if there exists a set  
 $\{z_g\in \cl{Z}(M):g\in G\}$  of central projections satisfying the conditions \eqref{c1}--\eqref{c3} and such that
\begin{equation}\label{Struct.1}
N=\ovl{{\mathrm{span}}}^{w^*}\{Mz_gg:g\in G\}.
\end{equation}

If this is the case, then $N$ is $B$-closed and there is a faithful normal  conditional expectation $E_N:\mxg\to N$. Moreover, $E_N$ is given by
 \begin{equation}\label{Struct.2}
E_N(x)=\sum_{g\in G} x_gz_gg, \ \ \ x\in \mxg,
\end{equation}
with convergence in the $B$-topology, where $\sum_{g\in G}x_gg$ is the Fourier series of $x$.
\end{thm}

\begin{proof} The necessity of conditions ({\ref{c1}})--({\ref{c3}}) for $N$ to be a von Neumann algebra is a matter of simple calculation. 

Suppose that $N$ is a von Neumann algebra lying between $M$ and $\mxg$. By
Lemma \ref{Mzg}, there exist central projections $\{z_g:g\in G\}$ such that
$Mz_g=\{x\in M:xg\in N\}$, from which it follows that \eqref{c1a}--\eqref{c3a}, and hence
\eqref{c1}--\eqref{c3}, are satisfied. If we let
$N_0=\ovl{{\mathrm{span}}}^{w^*}\{Mz_gg:g\in G\}$, then we see by Theorem
\ref{exp} that $N_0=N=\ovl{N}^B$, proving that $N$ has the desired form \eqref{Struct.1}. This theorem
also establishes that there is a faithful normal conditional expectation
$E_N:\mxg\to N$ and the formula for $E_N$ in \eqref{Struct.2} has already
been proved in \eqref{exp.7}, noting that $N=N_0$.
\end{proof}

We can now improve upon Lemma \ref{Bcont} (i).

\begin{cor}\label{NBcont}
Let $G$ be a discrete group acting on a von Neumann algebra $M$ by properly outer $*$-automorphisms and let $N$ be a von Neumann algebra satisfying
$M\subseteq N \subseteq \mxg$. Let $E:N\to M$ be the restriction of $E_M$ to $N$. If $\sigma$ is a $*$-automorphism of $N$ such that $\sigma(M)=M$, then
\begin{equation}\label{NBcont.1}
\sigma^{-1}\circ E\circ \sigma =E,
\end{equation}
and $\sigma$ is $B$-continuous.
\end{cor}

\begin{proof}
By Theorem \ref{Struct} there is a faithful normal conditional expectation $E_N : \mxg\to N$. Since $N'\cap(\mxg)\subseteq  
N$, this normal expectation is unique by \cite[Th\'{e}or\`{e}me 1.5.5(a)]{Con} and \eqref{NBcont.1} follows.

The proof that $\sigma$ is $B$-continuous is identical to the argument for the corresponding statement in Lemma \ref{Bcont}.
\end{proof}

\begin{rem} \quad (i) When $M$ is a factor, the central projections of
\eqref{c1}--\eqref{c3} are either 0 or 1, and these conditions easily imply
that $H:=\{g\in G:z_g=1\}$ is a subgroup of $G$. Thus we recapture the
result of \cite{Cho,ILP} for separable preduals that the intermediate factors in this case are all
of the form $\mxh$ where $H\subseteq G$ is a subgroup.

\noindent (ii)\quad From \eqref{Struct.2}, the map $\sum x_gg\mapsto \sum
x_gz_gg$ is a completely positive contraction, where the $z_g$'s satisfy
\eqref{c1}--\eqref{c3}. A simple proof of this would be interesting as we
have been unable to establish even the boundedness of this map by direct
methods.
\end{rem}

\section{Central groupoid normalizers}\label{groupoidsect}

 Let $\theta$ be a
properly outer $*$-automorphism of a countably decomposable von Neumann algebra $M$. Connes's noncommutative Rokhlin theorem uses \cite[Theorem 1.2.1]{Con}, which states that for each nonzero projection $q\in M$
and $\vp 
> 0$, there exists a nonzero projection $p\leq q$ such that 
$\| p\theta(p)\|<\vp$. This $*$-automorphism induces an action of $\mathbb Z$ on $M$ by $n\mapsto \theta^n$, $n\in \mathbb Z$, and there is a unitary $u\in
M\rtimes_\theta \mathbb{Z}$ that normalizes $M$ and implements $\theta$ as $x\mapsto uxu^*$, $x\in M$. The inequality $\| p\theta(p)\|<\vp$ is then equivalent to $\|
pup\|<\vp$. We also note that, in this situation, we have $E_M(u)=0$ (see Lemma \ref{Conlem2} below). Our first objective is to give a slightly more general version of the latter formulation that applies to certain groupoid normalizers of $M$ (see Definition \ref{groupoiddef}), and where we
drop the assumption of countable decomposability (Lemma \ref{Conlem3}). The relevance of groupoid normalizers to the types of problems considered here was established in \cite{Cam}. The most important result of the section is Lemma \ref{ideal}, which is used in Section \ref{Mercersect} to show that the natural candidates for the C$^*$-envelopes of certain operator spaces are indeed the correct ones (see Lemma \ref{env}).

Before proceeding, we note that the assumption of countable decomposability in \cite[Theorem 1.2.1]{Con} is unnecessary. 
Using \cite[Prop. II.3.19]{Tak1}, any nonzero projection $q\in M$ dominates a
nonzero countably decomposable projection $q_1\in M$ and the projection $e:=\vee \{\theta^n(q_1):n\in \mathbb{Z}\}$ is  countably decomposable and is $\theta	$-invariant. Then $eMe$ is a
countably decomposable von Neumann algebra on which $\theta$ acts as a properly outer $*$-automorphism, so the original version of  
\cite[Theorem 1.2.1]{Con}
can be applied to the pair $(\theta|_{eMe},q_1)$ to produce a nonzero projection $p\leq q_1\leq q$ so that $\| p\theta(p)\|<\vp$.

The next lemma is an immediate consequence of \cite[Lemme 1.5.6]{Con73} and so we omit the proof. We thank Ami Visalter for bringing this to our attention.

\begin{lem}\label{Conlem2}
Let $M\subseteq N$ be an inclusion of von Neumann algebras with a  conditional expectation $E:N\to M$, and suppose that $M'\cap N=\mathcal{Z}(M)$. Let $u\in N$ be a
unitary normalizer of $M$ and let $\theta$ be the $*$-automorphism $
\Ad u$ of $M$. Then $\theta$ is properly outer if and only if $E(u)=0$.
\end{lem}

\begin{lem}\label{Conlem3}
Let $M\subseteq N$ be an inclusion of von Neumann algebras such that $M'\cap N=\mathcal{Z}(M)$ and let $E:N\to M$ be a conditional expectation.
\begin{itemize}
\item[(i)] Let $v\in \gnmn$ satisfy $E(v)=0$. Given $\vp>0$ and a nonzero projection $g\in M$, there exists a nonzero projection $f\in M$ so that $f\leq g$ and $\|
fvf\|<\vp$.
\item[(ii)]
Let $v_1,\ldots ,v_k\in\gnmn$ satisfy $E(v_i)=0$, $1\leq i\leq k$. Given a nonzero projection $g\in M$ and elements $x_i,\ldots ,x_k\in M$, there exists a sequence
$\{e_n\}_{n=1}^\infty$ of nonzero projections in $M$ such that $g\geq e_1\geq e_2\ldots$ and
\begin{equation}
\lim_{n\to\infty} \| e_nx_iv_ie_n\|=0,\ \ \ 1\leq i\leq k.
\end{equation}
\end{itemize}
\end{lem}

\begin{proof}
Denote by $p$ and $q$ respectively the central projections $vv^*$ and $v^*v$, and let $\theta :M\to M$ be the $*$-homomorphism $\theta(x)=vxv^*$, $x\in M$. Then $\theta$
restricts to a $*$-isomorphism of $Mq$ onto $Mp$. Similarly the $*$-homomorphism $\phi(x)=v^*xv$, $x\in M$, restricts to a $*$-isomorphism of $Mp$ onto $Mq$ and is the
inverse of $\theta$.

If $qg=0$ then $gv^*vg=0$ and so $vg=0$. Then $\|gvg\|=0$ and we may take $f$ to be $g$. Thus we may assume that $qg\ne 0$ and, replacing $g$ by $qg$, we may also
assume that $g\leq q$. We now make a further reduction. Let $z\in \mathcal{Z}(M)$ be the central support projection of $g$ and let $w=vz\in \gnmn$. Then $E(w)=0$,
$w^*w=zqz=z$, and $g\leq z$. If there exists a nonzero projection $f\leq g$ so that $\| fwf\|<\vp$, then $\| fvzf\|<\vp$ and $fvzf=fvf$, so $\| fvf\|<\vp$. Thus, by
replacing $v$ by $vz$, we may further assume that $q$ is the central support projection of $g$. There are three cases to consider.

\noindent Case 1.\quad Suppose that $p=q$. Then $\theta$ restricts to a $*$-automorphism of $Mq$ and this is properly outer by Lemma \ref{Conlem2} since $E(v)=0$. By \cite[Theorem 1.2.1]{Con} there exists a nonzero projection $f\leq g$ so that $\| f\theta(f)\|<\vp^2$. Since
\begin{equation}
\| fvf\|^2=\| fvfv^*f\|=\|f\theta(f)f\|\leq \| f\theta(f)\|<\vp^2,
\end{equation}
we obtain $\| fvf\|<\vp$.

\noindent Case 2. \quad Suppose that $q(1-p)\ne 0$. Let $f=gq(1-p)$. Since $q(1-p)$ is a central projection below the central support $q$ of $g$, we see that $f\ne 0$,
and also that $f\leq g$. Then
\begin{equation}
fvf=fpvf=gq(1-p)pvf=0.
\end{equation}

\noindent Case 3. \quad Suppose that $(1-q)p\ne 0$. Put $f=g\phi((1-q)p)\leq g$ and observe that $f\ne 0$ since $\phi((1-q)p)$ is a central projection below the central
support $q$ of $g$. Then
\begin{equation}
fv^*f=gv^*v\phi((1-q)p)v^*g\phi((1-q)p)=
gv^*(1-q)pg\phi((1-q)p)=0
\end{equation}
since $\phi((1-q)p)\leq q$. Taking adjoints gives $fvf=0$.

These three cases exhaust the possibilities, completing the proof of (i).

Now consider elements $v_i\in \gnmn$ with $E(v_i)=0$, $x_i\in M$ and a projection $g\in M$. Since $M$ is spanned by its unitaries, there is no loss of generality in
assuming that each $x_i$ is a unitary  $u_i\in M$. Then $u_iv_i\in \gnmn$ with $E(u_iv_i)=0$, so replacing each $v_i$ with $u_iv_i$ allows us to make the further
assumption that $x_i=1$.

We will construct the nonzero projections $e_n$ inductively. Take $\vp=2^{-1}$ and apply part (i) with $v=v_1$ to obtain a projection $f_1\leq g$ with $\|
f_1v_1f_1\|<2^{-1}$. Then apply (i) again with $g$ and $v$ replaced respectively by $f_1$ and $v_2$ to find a projection $f_2\leq f_1$ and $\| f_2v_2f_2\|<2^{-1}$.
Continuing in this way, we find nonzero projections $g\geq f_1\geq f_2 \ldots \geq f_k$ in $M$ so that $\| f_iv_if_i\|<2^{-1}$ for $1\leq i\leq k$. Then take $e_1$ to be
$f_k$ and note that
\begin{equation}
\|e_1v_ie_1\|<2^{-1},
\ \ \ 1\leq i\leq k.
\end{equation}

Assuming that $g\geq e_1 \ldots \geq e_n$ have been constructed to satisfy
\begin{equation}
\| e_jv_ie_j\|<2^{-j},\ \ \ 1\leq i\leq k,\ \ \ 1\leq j\leq n,
\end{equation}
we repeat the argument with $e_n$ in place of $g$ and $\vp=2^{-(n+1)}$ to find $e_{n+1}\leq e_n$ satisfying
\begin{equation}
\| e_{n+1}v_ie_{n+1}\|<2^{-(n+1)},\ \ \ 1\leq i\leq k,
\end{equation}
completing the proof of part (ii).
\end{proof}

The hypotheses of the following three results have been chosen to be satisfied by the C$^*$-algebra C$^*(X_0)$ of Section \ref{Mercersect} (see equation \eqref{X0def}).

\begin{lem}\label{Conlem4}
Let $M\subseteq N$ be an inclusion of von Neumann algebras
with $M'\cap N=\mathcal{Z}(M)$, and let $A$ be a
C$^*$-algebra satisfying $M\subseteq A \subseteq N$.
\begin{itemize}
\item[(i)]
Suppose that $v\in A\cap \gnmn$. Then there exists a central
projection $z\in \mathcal{Z}(M)$ such that $zv\in M$ and
$E((1-z)v)=0$ for any conditional expectation $E:A\to M$.
\item[(ii)]
If $A$ is the norm closed span of $A\cap \gnmn$ and $E_1,E_2
:A\to M$ are conditional expectations, then $E_1=E_2$.
\item[(iii)]
If $A$ is the norm closed span of $A\cap \gnmn$, $\theta$ is
a $*$-automorphism of $A$ such that $\theta(M)=M$, and
$E:A\to M$ is a conditional expectation, then
$E=\theta^{-1}\circ E \circ \theta$.
\end{itemize}
\end{lem}
\begin{proof}
Let $v\in A\cap \gnmn$. By Zorn's lemma, there exists a
 central projection $z\in \mathcal{Z}(M)$  that is maximal with respect to the requirement that
$zv\in M$. Now let $E:A\to M$ be a conditional expectation,
let $w=(1-z)v\in A \cap \gnmn$, and put $y:=E(w)\in M$. To
arrive at a contradiction, suppose that $y\ne 0$.

Define a $*$-homomorphism $\phi :M\to M$ by
\begin{equation}
\phi(x)=wxw^*,\ \ \ x\in M.
\end{equation}
Since $z_1:=w^*w$ and $z_2:=ww^*$ are projections in $\cl{Z}(M)$, we see that $\phi$ restricts to a $*$-isomorphism of $Mz_1$ onto $Mz_2$. We have 
\begin{equation}\label{Conlem4eq1}
\phi(x)w=wxw^*w=wxz_1=wz_1x=wx,\ \ \ x\in M,
\end{equation}
so an application of $E$ gives
\begin{equation}\label{Conlem4eq2}
yx=\phi(x)y,
 \ \ \  x\in M.
\end{equation}
Thus $y^*y\in \cl{Z}(M)$. Moreover, since  $y^*=E(w^*)=E(w^*z_2)=y^*z_2$, we obtain  $yy^*\in (Mz_2)'\cap Mz_2=\cl{Z}(M)z_2$ from \eqref{Conlem4eq2}. Similarly,  $yz_1=y$, so the central support projection $p\in \cl{Z}(M)$ of $y$, and thus also of $y^*$, satisfies $p\leq z_1\wedge z_2$. Write $y=u|y|$ for the polar decomposition of $y\in pMp$ and note that $p$ is both the domain and range projection of $u$ so that $uu^*=u^*u=p$. From \eqref{Conlem4eq2},
\begin{equation}\label{Conlem4eq3}
ux|y|=u|y|x=yx=\phi(x)y=\phi(x)u|y|,\ \ \ x\in M,
\end{equation}
since $|y|\in \cl{Z}(M)$. If we define
\begin{equation}\label{Conlem4eq4}
J=\{t\in \cl{Z}(M)p:uxt=\phi(x)ut,\ x\in M\},
\end{equation}
then $J$ is a $w^*$-closed ideal in $\cl{Z}(M)p$ that contains $|y|$ and thus also its central support $p$. It follows that
\begin{equation}\label{Conlem4eq5}
upx=uxp=\phi(x)up,\ \ \ x\in M,
\end{equation}
and consequently that 
\begin{equation}\label{Conlem4eq6}
ux=\phi(x)u,\ \ \ x\in M.
\end{equation}
since $up=u$. Then, using \eqref{Conlem4eq6} and the adjoint of \eqref{Conlem4eq2},
\begin{equation}
w^*ux=w^*\phi(x)u=xw^*u,\ \ \ x\in M,
\end{equation}
so $w^*u\in M'\cap N=\cl{Z}(M)$. Multiply on the right by $u^*\in M$ to obtain $w^*p\in M$, and so $p(1-z)v=pw\in M$. Since 
\begin{equation}\label{Conlem4eq7}
pw=E(pw)=py=y\ne 0,
\end{equation}
we see that $p(1-z)\ne 0$ and $(z+p(1-z))v\in M$.
This  contradicts the maximality of $z$, so $E((1-z)v)=0$, proving part (i).

If $E_1,E_2:A \to M$ are conditional expectations then the first part  shows that they  agree on $A\cap \gnmn$ which 
spans a norm dense subspace of $A$. Thus $E_1=E_2$, proving part (ii).

Since $E$ and $\theta^{-1}\circ E\circ \theta$ are conditional expectations of $A$ onto $M$, equality follows from part (ii), establishing the last part of the lemma.
\end{proof}

\begin{lem}\label{ideal}
Let $M\subseteq N$ be an inclusion of von Neumann algebras with $M'\cap N=\mathcal{Z}(M)$ and let $E:N\to M$ be a faithful conditional expectation. Let $A$ be a
C$^*$-subalgbra of $N$ generated by $M$ and a subset of $\gnmn$. If $J$ is a norm closed nonzero ideal in $A$, then $M\cap J\ne\{0\}$.
\end{lem}
\begin{proof} Let $\pi :A\to A/J$ be the quotient map and let $j$ be a positive nonzero element of $J$. By the faithfulness of $E$, $E(j)\in M$ is a nonzero positive
element so we may multiply $j$ on the left by a suitable element $x\in M$ to obtain that $E(xj)$ is a nonzero projection $g\in M$. Since $\gnmn$ is closed under adjoints
and multiplication, $A$ has a norm dense $*$-subalgebra $B$ whose elements have the form
\begin{equation}\label{Bform}
x_0+\sum_{i=1}^k x_iv_i
\end{equation}
where $x_i\in M$ for $0\leq i \leq k$ and $v_i \in \gnmn$ for $1\leq i\leq k$. By Lemma \ref{Conlem4}, there exist central projections $z_i\in \mathcal{Z}(M)$ so that
$z_iv_i\in M$ and $E((1-z_i)v_i)=0$. Writing $x_iv_i=x_iz_iv_i+x_i(1-z_i)v_i$ 
and combining the $x_iz_iv_i$ terms with $x_0$ allow us to assume that we also have $E(v_i)=0$ in \eqref{Bform}. We may approximate $xj$ by
elements of $B$, so we may choose $x_i\in M$ for $0\leq i\leq k$ and $v_i\in A \cap \gnmn$ for $1\leq i \leq k$ with $E(v_i)=0$ so that
\begin{equation}
\| xj-x_0-\sum_{i=1}^k x_iv_i\|<\frac{1}{3}.
\end{equation}
Apply $E$ to this inequality to obtain $\| g-x_0\|<1/3$, so
\begin{equation}\label{twothirds}
\| xj-g-\sum_{i=1}^k x_iv_i\|<\frac{2}{3}.
\end{equation}
By Lemma \ref{Conlem3}, there are nonzero projections $g\geq e_1\geq e_2 \ldots $ in $M$ such that
\begin{equation}
\lim_{n\to\infty} \| e_nx_iv_ie_n\|=0,\ \ \ 1\leq i\leq k,
\end{equation}
so there is a sufficiently large choice of $n$ so that
$\|\sum_{i=1}^ke_nx_iv_ie_n\|<1/3$. Pre- and post-multiply by $e_n$ in \eqref{twothirds} to obtain the inequality
\begin{equation}
\| e_nxje_n -e_n\|<1.
\end{equation}
Since $e_nxje_n\in J$, we may apply the quotient map $\pi$ to obtain $\| \pi(e_n)\|<1$. However, $\pi(e_n)$ is a projection so 
$\pi(e_n)=0$ and $e_n\in J$. Thus $M\cap J\ne \{0\}$, as
required.
\end{proof}

\begin{cor}\label{autocont}
Let $M\subseteq N$ be an inclusion of von Neumann algebras with $M'\cap N=\cl{Z}(M)$ and a faithful conditional expectation $E_M:N\to M$. Let $A$ be a C$^*$-algebra satisfying $M\subseteq A \subseteq N$ and suppose that $A$ is the norm closed span of $A\cap \gnmn$. If $\theta $ is a $*$-automorphism of $A$ such that $\theta(M)=M$, then $\theta$ is $B$-continuous.
\end{cor}

\begin{proof}
Let $E:A\to M$ be the restriction of $E_M$ to $A$. Consider a net $\{x_\lambda\}_{\lambda\in \Lambda}$ in $ A$ such that $x_\lambda \to 0$ in the $B$-topology. If $\omega\in M_*$ is a fixed but arbitrary normal state on $M$, then
\begin{align}
\omega\circ E(\theta(x_\lambda)^*\theta(x_\lambda))&=\omega\circ E(\theta(x_\lambda^*x_\lambda))
=\omega\circ \theta\circ\theta^{-1}\circ E \circ \theta(x_\lambda^* x_\lambda)
=\omega\circ\theta\circ E(x_\lambda^* x_\lambda)
\end{align}
using Lemma \ref{Conlem4} (iii) for the last equality. Since the restriction of $\omega \circ \theta$ to $M$ is a normal state, it follows that
\begin{equation}
\lim_\lambda \omega\circ E(\theta(x_\lambda)^*\theta(x_\lambda))=\lim_\lambda (\omega\circ \theta)(E(x_\lambda^* x_\lambda))=0.
\end{equation}
Thus $\lim_\lambda \theta(x_\lambda)=0$ in the $B$-topology, showing that $\theta$ is $B$-continuous.
\end{proof}

\section{Norming of crossed products}\label{normingsect}

Recall from \cite{PSS} that if we have an inclusion $B\subseteq A$ of C$^*$-algebras, then $B$ is said to norm $A$ when the following is satisfied: for any integer $k$ and
matrix $X\in \mathbb{M}_k(A)$,
\begin{equation}
\|X\|=\sup \|RXC\|
\end{equation}
where the supremum is taken over row matrices $R$ and column matrices $C$ of length $k$ with entries from $B$ and satisfying $\|C\|,\|R\|\leq 1$. Since $RXC\in A$, the
point of this definition is to reduce the calculation of norms in $\mathbb{M}_k(A)$ to that of norms in $A$, and the concept of norming has proved useful in showing
complete boundedness of certain types of bounded maps; we will use it in this way to prove Lemma \ref{autocb}. We also recall that a C$^*$-algebra $A$ acting on a Hilbert space $\Hil$ is said to be locally cyclic if, given $\vp>0$ and vectors $\xi_1,\ldots,\xi_n\in \Hil$, there exist $\xi \in \Hil$ and $a_1,\ldots,a_n\in A$ such that $\|a_i\xi-\xi_i\|<\vp$ for $1\leq i\leq n$. When $\Hil$ is separable, this coincides with having a cyclic vector but is more general when $\Hil$ has higher cardinality.

We will need one result in this direction. As we point out in the proof, it is already known for all but one of the various types of von Neumann algebras, so our
objective is to put the previous cases under one general umbrella. The hypothesis of a properly outer action seems necessary. If $G$ acts trivially on $M=\mathbb{C}1$, then the crossed product is $L(G)$ which is not normed by $\mathbb{C}1$ unless $G$ is abelian.

\begin{thm}\label{norming}
Let $G$ be a discrete group acting on a von Neumann algebra $M$ by properly outer $*$-automorphisms. Then $M$ norms $\mxg$.
\end{thm}

\begin{proof}
The central projections that decompose $M$ into its components of various types are all invariant under the $\alpha_g$'s, so we may reduce to the case where $M$ is type
I$_1$, I$_\infty$, II$_1$, II$_\infty$, III, or I$_n$ for $2\leq n<\infty$. In the first five cases the result is already known; see \cite[Corollary 1.4.9]{CPZ} for type I$_1$,
\cite[Theorem 2.9]{PSS} for types I$_\infty$, II$_\infty$ and III, and \cite[Remark 2.5(iii)]{PS} for type II$_1$ where the argument given is for factors but applies generally without change. Thus
it suffices to consider a type I$_n$ algebra $\mathbb{M}_n\otimes A$ where $A$ is an abelian von Neumann algebra. We work first with the special case where $G$ is
countable and $A$ is countably decomposable. Under these assumptions, we can fix a faithful normal state $\phi$ on $A$ \cite[Prop. II.3.19]{Tak1}, and we let $\tau$ denote the unital
trace on $\mathbb{M}_n$ whereupon $\tau\otimes \phi$ is a  faithful normal state on $M$.

The center of $M$ is $A$, so the $*$-automorphisms $\alpha_g$ leave $A$ invariant. If $\beta_g$ denotes the restriction to $A$ of $\alpha_g$, then we have a second
crossed product $A\rtimes_\beta G\subseteq \mxg$. Suppose that there exists $g\in G\setminus\{e\}$ so that $\beta_g$ is not properly outer on $A$. Then there exists a
nonzero projection $p\in A$ so that $\beta_g|_{Ap}=\mathrm{id}$. It follows that $\alpha_g$ leaves $\mathbb{M}_n\otimes Ap$ invariant and fixes each element of the
center $Ap$. By \cite[Corollary 9.3.5]{KR2} $\alpha_g$ is inner on $\mathbb{M}_n\otimes Ap$, contradicting the proper outerness of $\alpha_g$ on $M$. Thus each $\beta_g$
is properly outer for $g\in G\setminus \{e\}$, and so $A$ is a Cartan masa in $A\rtimes_\beta G$.

Let $\psi$ be the faithful normal state on $\mxg$ defined by $\psi=(\tau \otimes\phi)\circ E_M$. The Hilbert space $\Hil_0$ in the standard representation of
$A\rtimes_\beta G$ is the completion of $A\rtimes_\beta G$ with the inner product
\begin{equation}
\langle x,y\rangle=\phi(y^*x),\ \ \ x,y\in A\rtimes_\beta G.
\end{equation}
In the development of modular theory \cite[{\S}9.2]{KR2} (see also \cite{Str}) there is a preclosed conjugate linear map $S_0:A\rtimes_\beta G\to \Hil_0$ defined by
\begin{equation}
S_0(x)=x^*,\ \ \ x\in A\rtimes_\beta G,
\end{equation}
and its closure is denoted $\ovl{S_0}$. The polar decomposition is $\ovl{S_0}=J_0(\ovl{S_0}^*\,\ovl{S_0})^{1/2}$ where $J_0:\Hil_0\to \Hil_0$ is a surjective conjugate linear
isometry. If we carry out the same construction for $\mxg$ using the state $\psi$, then we obtain a corresponding Hilbert space $\Hil_1$ and operators $S_1$ and $J_1$. We
may identify $\Hil_0$ with the subspace of $\Hil_1$ defined by the operators in $(I_n\otimes A)\rtimes_\beta G$, and since $S_1|_{\Hil_0}=S_0$, we see that $J_1|_{\Hil_0}=J_0$.
It follows that $J_1(A\rtimes_\beta G)J_1$ restricted to $\Hil_0$ is $J_0(A\rtimes_\alpha G)J_0$.

As noted above, $A$ is a Cartan masa in $A\rtimes_\beta G$ so by \cite[Theorem 1.4.7]{CPZ} (see also \cite[Proposition 2.9 (1)]{FM2} for the separable predual case) $\{A\cup J_0AJ_0\}''$ is a masa in $\cl{B}(\Hil_0)$ so is locally cyclic on $\Hil_0$ \cite[Theorem 2.7]{PSS}. Since $I_n$ is a cyclic vector for
$\mathbb{M}_n$ in its standard representation from $\tau$, it follows that $\{M\cup J_1AJ_1\}''=\{(\mathbb{M}_n\otimes A)\cup J_1AJ_1\}''$ is locally cyclic on $\Hil_1$.
Thus we have found an abelian von Neumann algebra $B:=J_1AJ_1\subseteq (\mxg)'$ so that $\{M \cup B\}''$ is locally cyclic, and so $M$ norms $\mxg$ by \cite[Theorem 2.7]{PSS}.

For the general case, let $G_F$ be the countable subgroup of $G$ generated by a finite subset $F$ of $G$ and let $E_F:\mxg\to \fmxg$ denote the normal conditional
expectation. If we view the collection $\cl F$ of finite subsets as an upwardly directed net, then $\{(I_k\otimes E_F)(X)\}$ is a uniformly bounded net converging in the
$B$-topology  to $X$ for $X\in \mathbb{M}_k\otimes(\mxg)$ and $k$ arbitrary. Moreover, by Lemma \ref{Bconv}, convergence also takes place in the $w^*$-topology. Then $\limsup_F\|(I_k\otimes E_F)(X)\|=\|X\|$, so a straightforward
calculation shows that if $M$ norms each $\fmxg$, then it also norms $\mxg$. Thus it suffices to assume that $G$ is countable.

If $p\in A$ is a countably decomposable projection then so is $q:=\vee\{\alpha_g(p):g\in G\}$, since $G$ is countable, and $q$ is $G$-invariant by construction. A
maximality argument then gives a family $\{q_\lambda\}_{\lambda\in \Lambda}$ of orthogonal $G$-invariant countably decomposable projections in $A$ that sum to 1. These
decompose $\mxg$ as $\oplus_{\lambda\in \Lambda}(Mq_\lambda \rtimes_\alpha G)$ and the special case above applies to each summand. Thus $M$ norms $\mxg$ as required.
\end{proof}

\section{Mercer's  extension theorem}\label{Mercersect}
 In this section we consider an inclusion $M\subseteq X \subseteq \mxg$ where $X$ is a $w^*$-closed $M$-bimodule. Subsequently we will use the notation $W^*(\cdot)$ to denote the von Neumann algebra generated by a set of operators.
Our objective is Theorem \ref{liftX0}, where we show that suitable $M$-bimodule maps on $X$ extend to $*$-automorphisms of $W^*(X)$, a version of Mercer's  extension theorem \cite{Mer} in the setting of crossed products.  We will need two auxiliary sets defined as follows:
\begin{equation}\label{X0def}
X_0=
{\overline{{\mathrm{span}}_M}}^{\|\cdot\|}\{X\cap \gncp\},
\end{equation}
the norm closed $M$-linear span of $X\cap \gncp$, and
\begin{equation}\label{Ydef}
Y=\{E_M(xg^{-1})g:x\in X,\ g\in G\},
\end{equation}
the set of elements $mg$ that appear in the Fourier series of some $x\in X$. It follows from Lemma \ref{Mzg} that 
\begin{equation}\label{YX0X}
Y\subseteq X_0\subseteq X.
\end{equation}

\begin{lem}\label{generate}
Let a discrete group $G$ act on a von Neumann algebra $M$ by properly outer $*$-automorphisms and let $X$ be a $w^*$-closed $M$-bimodule satisfying $M\subseteq X\subseteq \mxg$. Then
\begin{equation}
W^*(Y)=W^*(X_0)=W^*(X),
\end{equation}
where $X_0$ and $Y$ are as defined in \eqref{X0def} and \eqref{Ydef}.

In particular, if $X$ generates $\mxg$ as a von Neumann algebra, then this crossed product is also generated by $X_0$ and by $Y$.
\end{lem}

\begin{proof}
The inclusion  $W^*(Y)\subseteq W^*(X)$ is immediate from the containment \eqref{YX0X}.

Since $M\subseteq X$, \eqref{Ydef} gives that $M\subseteq Y$ and so $W^*(Y)$ is a von Neumann algebra lying between $M$ and $\mxg$. By Theorem \ref{exp}, $W^*(Y)$ is $B$-closed and so $\overline{{\mathrm{span}}}^B\{Y\}\subseteq W^*(Y)$. Consider an element $x=\sum_{g\in G}x_gg\in X$. Then $x_gg\in Y$ for each $g\in G$ from \eqref{Ydef}, and since the Fourier series converges in the $B$-topology, we conclude that $x\in \overline{{\mathrm{span}}}^B\{Y\}\subseteq W^*(Y)$. This shows that $W^*(X)\subseteq W^*(Y)$, establishing the reverse containment.
\end{proof}

Our aim now is to show that certain maps on $M$-bimodules $X$ extend to $*$-automorphisms of $W^*(X)$. To this end, we say that a map $\theta:X\to X$ is an $M$-bimodule map if the following two conditions are satisfied:
\begin{itemize}
\item[(i)] The restriction of $\theta$ to $M$ is a $*$-automorphism.
\item[(ii)] For $x\in X$ and $m_1,\,m_2\in M$,
\begin{equation}\label{mapdef}
\theta(m_1xm_2)=\theta(m_1)\theta(x)\theta(m_2).
\end{equation}
\end{itemize}

\begin{lem}\label{autocb}
Let a discrete group $G$ act on a von Neumann algebra $M$ by properly outer $*$-automorphisms and let $X$ be a $w^*$-closed $M$-bimodule satisfying $M\subseteq X\subseteq \mxg$. Let $\theta :X\to X$ be an isometric surjective $M$-bimodule map. Then $\theta$ is a complete isometry.
\end{lem}

\begin{proof} By Theorem \ref{norming}, $M$ norms $\mxg$. The result now follows by using the argument of  \cite[Theorem 1.4]{Pitts}
(see also \cite[Theorem 2.10]{PSS} or \cite[Theorem 2.1]{Sm} for earlier versions of this type of result).
\end{proof}

\begin{lem}\label{thetaX0}
Let a discrete group $G$ act on a von Neumann algebra $M$ by properly outer $*$-automorphisms and let $X$ be a $w^*$-closed $M$-bimodule satisfying $M\subseteq X\subseteq \mxg$. Let $X_0$ be defined as in \eqref{X0def}, and let $\theta :X\to X$ be an isometric surjective $M$-bimodule map. Then $\theta $ maps $X_0$ onto $X_0$.
\end{lem}

\begin{proof}
Consider a  fixed but arbitrary $v\in X\cap\gncp$. To see that $\theta$ maps $X_0$ into $X_0$, it suffices to show that $\theta(v)\in X\cap\gncp$. Define $\beta:M\to M$ by $\beta(x)=vxv^*$ for $x\in M$. Let $z_1=v^*v$ and $z_2=vv^*$ be central projections in $M$ and observe that $\beta$ is a $*$-homomorphism that restricts to a $*$-isomorphism of $Mz_1$ onto $Mz_2$. Moreover,
\begin{equation}
\beta(x)v=vxv^*v=vxz_1=vz_1x=vx,\ \ x\in M,
\end{equation}
and an application of $\theta$ gives
\begin{equation}
\theta(\beta(x))\theta(v)=\theta(v)\theta(x),\ \ \ x\in M.
\end{equation}
Replacing $x$ by $\theta^{-1}(x)$ leads to 
\begin{equation}
\theta(v)x=\theta(\beta(\theta^{-1}(x)))\theta(v),\ \ \ x\in M,
\end{equation}
from which it follows that
\begin{equation}
\theta(v)^*\theta(v)\in M'\cap \mxg=\Z(M)
\end{equation}
and that
\begin{equation}
\theta(v)\theta(v)^*\in \theta(\beta(\theta^{-1}(x)))'\cap(\theta(z_2)(\mxg))=\Z(M)\theta(z_2)\subseteq \Z(M).
\end{equation}
For $\vp\in (0,1/2)$, let $p_\vp$ be the spectral projection of $\theta(v)\theta(v)^*$ for the interval $[\vp,1-\vp]$. Then $\|p_\vp\theta(v)\|<1$ so $\|\theta^{-1}(p_\vp)v\|<1$. Also $\|\theta^{-1}(p_\vp)v\|^2=\|\theta^{-1}(p_\vp)vv^*\|$ and this last term is a central projection which is consequently 0, otherwise its norm would be 1. Thus $p_\vp \theta(v)\theta(v)^*=0$, and it follows that $\theta(v)\theta(v)^*$ is a projection by letting $\vp\to 0$. Then $\theta(v)^*\theta(v)$ is also a projection, and it has already been shown that these elements lie in $\Z(M)$. Since
\begin{equation}
\theta(v)\theta(x)\theta(v)^*=\theta(vx)\theta(v)^*=\theta(vxv^*v)\theta(v)^*=\theta(vxv^*)\theta(v)\theta(v)^*,\ \ \ x\in M,
\end{equation}
we see that $\theta(v)M\theta(v)^*\subseteq M$, and a similar calculation shows that $\theta(v)^*M\theta(v)\subseteq M$. Thus 
$\theta(v)\in \gncp$, and it follows that $\theta(X_0)\subseteq X_0$. The same argument applied to $\theta^{-1}$ gives $\theta^{-1}(X_0)\subseteq X_0$, proving that $\theta(X_0)=X_0$.
\end{proof}

Recall that a C$^*$-algebra $A$ is said to be the C$^*$-{\emph{envelope}}
of a unital operator space $X$ if there is a completely isometric unital embedding $\iota':X\to A$ so that $\iota'(X)$ generates $A$, and if $B$ is another C$^*$-algebra with a completely
isometric unital embedding $\iota:X\to B$ whose range generates $B$, then there is a $*$-homomorphism $\pi:B\to A$ so that $\pi\circ\iota=\iota'$ (which entails surjectivity of $\pi$).
Every unital operator space has a unique C$^*$-envelope denoted C$^*_{env}(X)$, \cite{Ar1,Ar2,Ha}.

\begin{lem}\label{env}
Let a discrete group $G$ act on a von Neumann algebra $M$ by properly outer $*$-automorphisms and let $X$ be a $w^*$-closed $M$-bimodule satisfying $M\subseteq X\subseteq \mxg$. Let $X_0$ be defined as in \eqref{X0def}. Then C$^*(X_0)$ is the C$^*$-envelope of $X_0$.
\end{lem}

\begin{proof}
Let $\iota:X_0\to C^*(X_0)$ be the identity embedding and let $\iota':X_0\to C^*_{env}(X_0)$ be a completely isometric embedding such that $\iota'(X_0)$ generates $C^*_{env}(X_0)$. By the definition of the C$^*$-envelope, there is a surjective $*$-homomorphism $\pi:
C^*(X_0)\to C^*_{env}(X_0)$ such that $\pi\circ \iota=\iota'$ on $X_0$. Letting $J=\ker \pi$, it suffices to show that $J=\{0\}$.

Now C$^*(X_0)$ satisfies the hypothesis for the C$^*$-algebra $A$ in Lemma \ref{ideal}, so if $J\ne\{0\}$, then $J\cap M\ne\{0\}$. This would imply that $\iota'|_M$ has a nontrivial kernel, a contradiction. Thus $J=\{0\}$ as required. 
\end{proof}

The next lemma addresses a technical point that will occur several times in the proof of Theorem \ref{liftX0}.

\begin{lem}\label{techlem}
Let a discrete group $G$ act on a von Neumann algebra $M$ by properly outer $*$-automorphisms. Let $\{x_\lambda\}_{\lambda\in\Lambda}$ be a uniformly bounded net in $M$ converging strongly to $x\in M$. Then
\begin{equation}\label{techlem1}
B{\text{-}}\lim_\lambda x_\lambda g=xg,\ \ \ \ g\in G.
\end{equation}
\end{lem}

\begin{proof}
Assume that $M$ is faithfully represented on a Hilbert space $\Hil$ and let $\xi\in\Hil$ be a fixed unit vector with associated vector state $\omega_\xi(\cdot)=\langle \cdot\,\xi,\xi\rangle$. Then
\begin{equation}\label{techlem2}
\lim_\lambda \omega_\xi\circ E_M((xg-x_\lambda g)^*(xg-x_\lambda g))=\lim_\lambda\langle
\alpha_{g^{-1}}(x-x_\lambda)^*\alpha_{g^{-1}}(x-x_\lambda)\xi,\xi\rangle=0
\end{equation}
since $*$-automorphisms are strongly continuous on uniformly bounded sets. The convex hull of the vector states is norm dense in the set of normal states of $M$, so \eqref{techlem1} follows from \eqref{techlem2}.
\end{proof}

We come now to the main result of this section, the extension of certain maps on bimodules to $*$-automorphisms of the generated von Neumann algebras.

\begin{thm}\label{liftX0}
Let a discrete group $G$ act on a von Neumann algebra $M$ by properly outer $*$-automorphisms and let $X$ be a $w^*$-closed $M$-bimodule satisfying $M\subseteq X\subseteq \mxg$. Let $\theta:X\to X$ be a $w^*$-continuous surjective isometric $M$-bimodule map. Then there exists a $*$-automorphism $\phi$ of $W^*(X)$ such that $\phi |_{X}=\theta$.
\end{thm}
\begin{proof}

Let $X_0$ be defined as in \eqref{X0def}.
By Lemma \ref{autocb}, $\theta$ is a complete isometry on $X$ and maps $X_0$ onto itself using Lemma \ref{thetaX0}. By definition of the 
C$^*$-envelope, which is C$^*(X_0)$ from Lemma \ref{env}, the restriction of $\theta$ to $X_0$ extends to a $*$-automorphism $\ovl{\theta}$ 
of C$^*(X_0)$. We now extend $\ovl{\theta}$ to a  map $\phi$ on $W^*(X_0)$ which we will then prove is a $*$-automorphism of this von Neumann algebra.

Consider $x\in W^*(X_0)$ and let $K$ be the norm closed ball in C$^*(X_0)$ of radius $\|x\|$. By the Kaplansky density theorem, $x\in \ovl{K}^{w^*}$. By \cite[Remark
3.3(i)]{CaSm}, the $w^*$- and $B$-closures of $K$ coincide, so there exists a net $\{x_\lambda\}_{\lambda\in \Lambda}$ from $K$  with limit $x$ in the $B$-topology. By Lemma
\ref{Bconv}, we also have $w^*	\text{-}\lim_\lambda x_\lambda =x$. By dropping to a subnet, we may further assume that 
$w^*\text{-}\lim_\lambda \ovl{\theta}(x_\lambda)$ exists in
$W^*(X_0)$, and so we have found a uniformly bounded net $\{x_\lambda\}_{\lambda\in\Lambda}\in C^*(X_0)$ such that
\begin{equation}
B	\text{-}\lim_\lambda x_\lambda =w^*	\text{-}\lim_\lambda x_\lambda =x
\end{equation}
while $w^*	\text{-}\lim_\lambda\ovl{\theta}(x_\lambda)$ exists. If $\{x_{\lambda}'\}_{\lambda\in \Lambda}$ is another such net with these properties, then
$B	\text{-}\lim_{\lambda}(x_\lambda-x_\lambda')=0$ so $B	\text{-}\lim_\lambda(\ovl{\theta}(x_\lambda)-\ovl{\theta}(x_\lambda'))=0$ from Corollary \ref{autocont} which applies since C$^*(X_0)$ satisfies the hypothesis. By Lemma \ref{Bconv}, we
also have convergence to 0 in the $w^*$-topology, so there is a well-defined linear contraction $\phi:W^*(X_0)\to W^*(X_0)$ defined by
\begin{equation}
\phi(x)=w^*	\text{-}\lim_{\lambda}(\ovl{\theta}(x_\lambda)),\ \ \ x\in W^*(X_0).
\end{equation}
Our next objective is to obtain \eqref{phiform} which gives a formula for the $g$-coefficients of $\phi(x)$.

Since $W^*(X_0)$ is a von Neumann algebra lying between $M$ and $\mxg$, Theorem \ref{Struct} gives a set of central projections $\{z_g:g\in G\}$ so that
\begin{equation}
W^*(X_0)=\ovl{\mathrm{span}}^{w^*}\{Mz_gg:g\in G\}.
\end{equation}
Let $Y$ be defined as in \eqref{Ydef} and, for each $g\in G$, let $K_g\subseteq M$ be the algebraic ideal given by
\begin{equation}
K_g=\{m\in M:mg\in \mathrm{Alg}(Y,Y^*)\}\subseteq Mz_g.
\end{equation}
Suppose that there exists $g_0\in G$ so that $\ovl{K_{g_0}}^{w^*}\ne Mz_{g_0}$. Since $\ovl{K_{g_0}}^{w^*}=Mp_{g_0}$ for some central projection $p_{g_0}\in \Z(M)$, we see
that $(z_{g_0}-p_{g_0})K_{g_0}=\{0\}$. From the definition of $Y$, each element in $\mathrm{Alg}(Y,Y^*)$ has a Fourier series of finite length and each element appearing
in one of these Fourier series itself lies in $\mathrm{Alg}(Y,Y^*)$. Consequently\begin{equation}
\mathrm{Alg}(Y,Y^*)\subseteq\sum_{g\in G}K_gg
\end{equation}
and so, for each $x\in \mathrm{Alg}(Y,Y^*)$,
\begin{equation}
E_M(x\alpha_{g_0^{-1}}(z_{g_0}-p_{g_0})g_0^{-1})=0.
\end{equation}
Since $\mathrm{Alg}(Y,Y^*)$ is $w^*$-dense in $W^*(X_0)$ from Lemma \ref{generate}, we see that
\begin{equation}
M(z_{g_0}-p_{g_0})=E_M(Mz_{g_0}g_0\alpha_{g_0^{-1}}(z_{g_0}-p_{g_0})g_0^{-1})=\{0\},
\end{equation}
implying that $p_{g_0}=z_{g_0}$, a contradiction. Thus each $K_g$ is $w^*$-dense in $Mz_g$.

Fix $g\in G$. The ideal $\ovl{K_g}^{\|\cdot\|}$ is also $w^*$-dense in $Mz_g$, so has an increasing approximate identity $\{e_\gamma\}_{\gamma\in \Gamma}$ converging
strongly to $z_g$. For $m\in M$, we have $me_\gamma,e_\gamma g\in C^*(X_0)$, so multiplicativity of $\ovl{\theta}$ on C$^*(X_0)$ gives
\begin{equation}\label{eqA}
\theta(me_\gamma)\ovl{\theta}(e_\delta g)=\ovl{\theta}(me_\gamma e_\delta g),\ \ \ \ m\in M, \ \ \gamma, \delta \in \Gamma.
\end{equation}
Now, by Lemma \ref{techlem},  $B\text{-}\lim_\delta e_\delta g=z_gg$ and, by dropping to a subnet, we may assume that $w^*	\text{-}\lim_\delta
\ovl{\theta}(e_\delta g)=\phi(z_gg)$. Then, for each $\gamma\in \Gamma$, Lemma \ref{techlem} and \eqref{eqA} respectively imply that
\begin{equation}\label{eqB}
B\text{-}\lim_\delta me_\gamma e_\delta g= me_\gamma g\ \ \ \mathrm{and}\ \ \ w^*	\text{-}\lim_\delta \ovl{\theta}(me_\gamma e_\delta g)=\theta(me_\gamma)\phi(z_gg).
\end{equation}
Thus
\begin{equation}\label{eqC}
\ovl{\theta}(me_\gamma g)=\theta(me_\gamma)\phi (z_gg),\ \ \ m\in M,\ \gamma\in \Gamma,
\end{equation}
using the $B$-continuity of $\ovl{\theta}$ on C$^*(X_0)$ (Lemma \ref{autocont}).
Then Lemma \ref{techlem} and \eqref{eqC} give respectively
\begin{equation}\label{eqD}
B	\text{-}\lim_\gamma me_\gamma g=mz_gg\ \ \ \mathrm{and}
\ \ \ w^*	\text{-}\lim_\gamma \ovl{\theta}(me_\gamma g)=\theta(mz_g)\phi(z_gg)
\end{equation}
so, by \eqref{eqD} and the definition of $\phi$, 
\begin{equation}\label{formula}
\phi(mz_gg)=\theta(mz_g)\phi(z_gg)=\theta(m)\theta(z_g)\phi(z_gg),\ \ \ m\in M.
\end{equation}
Putting $m=1$ gives $\phi(z_gg)=\theta(z_g)\phi(z_gg)$, so \eqref{formula} becomes
\begin{equation}
\phi(mz_gg)=\theta(m)\phi(z_gg),\ \ \ m\in M.
\end{equation}
In exactly the same way, we obtain a contractive extension $\sigma: W^*(X_0)\to W^*(X_0)$ of $\theta^{-1}|_{X_0}$ and
\begin{equation}
\sigma(mz_gg)=\theta^{-1}(m)\sigma(z_gg),
\ \ \ m\in M.
\end{equation}

Fix $x\in W^*(X_0)$ and a uniformly bounded net $\{x_\lambda\}_{\lambda\in\Lambda}\in C^*(X_0)$ so that
\begin{equation}
B	\text{-}\lim_\lambda x_\lambda =x\ \mathrm{and}\
w^*	\text{-}\lim_\lambda \ovl{\theta}(x_\lambda)=\phi(x).
\end{equation}
Fix $g\in G$ and let $\{f_\omega\}_{\omega\in \Omega}$ be an increasing approximate identity for ${\ovl{K_{g^{-1}}}}^{\|\cdot\|}$, whereupon $f_\omega \to z_{g^{-1}}$ strongly. If $\phi(x)_g$
denotes the $g$-coefficient in the Fourier series of $\phi(x)$, then
\begin{align}
\phi(x)_g\alpha_g(f_\omega)&= E_M(\phi(x)g^{-1}\alpha_g(f_\omega))=E_M(\phi(x)f_\omega g^{-1})\notag\\
&=w^*	\text{-}\lim_\lambda E_M(\ovl{\theta}(x_\lambda)f_\omega g^{-1})
=w^*	\text{-}\lim_\lambda E_M(\ovl{\theta}(x_\lambda \ovl{\theta}^{-1}(f_\omega g^{-1})))\notag\\
&=w^*	\text{-}\lim_\lambda {\theta}(E_M(x_\lambda \ovl{\theta}^{-1}(f_\omega g^{-1})))
=\theta(E_M(x\theta^{-1}(f_\omega)\sigma(z_{g^{-1}}g^{-1}))),
\end{align}
where we have applied Lemma \ref{Conlem4} (iii) to C$^*(X_0)$ for the fifth equality, and used the relation 
$\ovl{\theta}^{-1}(f_\omega g^{-1})=\sigma(f_\omega z_{g^{-1}}g^{-1})=\ovl{\theta}^{-1}(f_\omega)\sigma(z_{g^{-1}}g^{-1})$ for the last equality.
Now take the limit over $\omega$ to obtain
\begin{align}
\phi(x)_g&= \phi(x)_g z_g=\phi(x)_g\alpha_g(z_{g^{-1}})=
w^*	\text{-}\lim_\omega \phi(x)_g\alpha_g(f_\omega)\notag\\
&= w^*	\text{-}\lim_\omega \theta(E_M(x\theta^{-1}(f_\omega)\sigma(z_{g^{-1}}g^{-1})))
=\theta (E_M(x\theta^{-1}(z_{g^{-1}})\sigma(z_{g^{-1}}g^{-1})))\notag\\
&=\theta(E_M(x\sigma(z_{g^{-1}}g^{-1}))),\ \ \ \ x\in W^*(X_0),\ \ \ \ g\in G, \label{phiform}
\end{align}
where we have employed \eqref{c2} to replace $z_g$ by $\alpha_g(z_{g^{-1}})$ in the second equality.

We now use this formula to check the $w^*$-continuity of $\phi$. By the Krein--Smulian theorem, it suffices to consider a uniformly bounded net $\{x_\mu\}$ from
$W^*(X_0)$ such that
\begin{equation}
w^*\text{-}\lim_\mu x_\mu =0\ \ \ \mathrm{and}\ \ \
w^* 	\text{-}\lim_\mu \phi(x_\mu)=y\in W^*(X_0)
\end{equation}
and to conclude that $y=0$. For each $g\in G$, \eqref{phiform} gives
\begin{align}
E_M(yg^{-1})&=w^*		\text{-}\lim_\mu E_M(\phi(x_\mu)g^{-1})=w^*	\text{-}\lim_\mu \phi(x_\mu)_g\notag\\
&=w^*	\text{-}\lim_\mu \theta(E_M(x_\mu \sigma(z_{g^{-1}}g^{-1})))=0,
\end{align}
since $w^*	\text{-}\lim_\mu x_\mu =0$. Thus $y=0$, and $w^*$-continuity of $\phi$ is established. In the same way $\sigma$ is $w^*$-continuous, and $\phi\circ
\sigma=\sigma\circ\phi={\mathrm{id}}$ on $W^*(X_0)$ since these relations hold on C$^*(X_0)$. It is then clear that $\phi$ is a $*$-automorphism of $W^*(X_0)=W^*(X)$, where this equality is Lemma \ref{generate}. At this point we have only established that $\phi$ and $\theta$ agree on $X_0$, and it may be that $X_0$ is not $w^*$-dense in $X$. Thus the final step of proving equality of these two maps on $X$ requires further argument.

From Corollary \ref{NBcont}, both $\phi$ and $\sigma$ are $B$-continuous on $W^*(X)$, so if we define $w_g=\phi(z_gg)$ and apply $\phi$ to the Fourier series $\sigma(x)=\sum_{g\in G}y_gz_gg$ for an element $x\in W^*(X)$, then we obtain a $B$-convergent series $x=\sum_{g\in G}\phi(y_g)w_g$. This gives a second type of Fourier expansion for elements of $W^*(X)$.

Suppose that $x=\sum_{g\in G}x_gz_gg\in X$ is such that $\phi(x)\ne \theta(x)$. Then $\phi(x)=\sum_{g\in G} \phi(x_g)w_g$, and let $\sum_{g\in G} y_gw_g$ be the corresponding expansion for $\theta(x)\in X$. Fix $h\in G$ so that $\phi(x_h)w_h\ne y_hw_h$. Since $\phi$ and $\theta$ agree on $X_0$, we may replace $x$ by $x-x_hz_hh$, whereupon the $h$-term in the expansion of $\theta(x-x_hz_hh)=\theta(x)-\phi(x_h)w_h$ is $y_hw_h-\phi(x_h)w_h\ne 0$. Thus we may assume that $x_hz_h=0$ while $y_hw_h\ne 0$. Multiplication on the left by a suitable element of $M$ allows us to assume that $y_h$ is a projection $p$ with $pw_h\ne 0$, and the $h$-coefficient of $x$ is still 0. The final reduction is to replace $p$ by a central projection.

From \cite[p. 61]{SS2} we may choose partial isometries $v_i$ such that $\sum_i v_ipv_i^*$ is the central support projection of $p$. Now define $\tilde{x}=\sum_i\sigma(v_i)x\alpha_{h^{-1}}(\sigma(v^*_i))$. The $w^*$-continuity of $\theta$ gives $\theta(\tilde{x})=\sum_iv_i\theta(x)\phi(\alpha_{h^{-1}}(\sigma(v_i^*)))$ and the $h$-term in the expansion of $\theta(\tilde{x})$ is
\begin{align}
\sum_i v_iy_h\phi(z_hh)\phi(\alpha_{h^{-1}}(\sigma(v_i^*)))
&=\sum_iv_ip\phi(z_hh\alpha_{h^{-1}}(\sigma(v_i^*)))\notag\\
&=\sum_i v_ip\phi(\sigma(v_i^*)z_hh)
=\sum_iv_ipv_i^*w_h.
\end{align}
Noting that the $h$-coefficient of $\tilde{x}$ is 0, we may replace $x$ by $\tilde{x}$ to arrive at the following situation: $x_h=0$ while the $h$-term $y_hw_h$ in the expansion of $\theta(x)$ is $pw_h\ne 0$ where $p$ is a central projection in $M$.

Let $\beta=(m_j)$ be the net of collections of operators from Lemma \ref{Mzg} and define
\begin{equation}
x^{\beta}=\sum_j m_j^*x\alpha_{h^{-1}}(m_j).
\end{equation}
The $h$-coefficient of $x^\beta$ is 0, while the $g$-coefficient for $g\ne h$ is $\sum_jm_j^*x_g\alpha_{gh^{-1}}(m_j)$. For each of these sums, the $w^*$-limit over $\beta$ is 0 by Remark \ref{outer} since $\alpha_{gh^{-1}}$ is properly outer. 
This is a uniformly bounded net in $X$ so by dropping to a subnet, we may assume that it has a $w^*$-limit which must therefore be 0. Thus $w^*\text{-}\lim_\beta \theta(x^\beta)=0$.

Since
\begin{equation}
\theta(x^\beta)=\sum_j\phi(m_j^*)\sum_{g\in G}y_gw_g\phi(\alpha_{h^{-1}}(m_j)),
\end{equation}
the $h$-term in the expansion of $\theta(x^\beta)$ is
\begin{align}
\sum_j\phi(m_j^*)pw_h\phi(\alpha_{h^{-1}}(m_j))
&=\sum_j\phi(m_j^*\sigma(p)z_hh\alpha_{h^{-1}}(m_j))\notag\\
&=\phi(\sum_jm_j^*m_j\sigma(p)z_hh)=pw_h\ne 0,
\end{align}
using the $w^*$-continuity of $\phi$
and the centrality of $\sigma(p)z_h$.
 From this, we reach the contradiction that $w^*\text{-}
\lim_\beta \theta(x^\beta)\ne 0$. Thus $\phi(x)=\theta(x)$ for all $x\in X$, completing the proof.
\end{proof}

\begin{rem}
(i) \quad There is a more general version of Theorem \ref{liftX0} that is easily deduced from it. Let discrete groups $G_1$ and $G_2$ act on a von Neumann algebra $M$ by properly
outer $*$-automorphisms $\{\alpha_{g_1}:g_1\in G_1\}$ and $\{\beta_{g_2}:g_2\in G_2\}$ respectively. For $i=1,2$ let $X_i$ be $w^*$-closed $M$-bimodules satisfying
$M\subseteq X_1\subseteq M\rtimes_\alpha G_1$ and $M\subseteq X_2\subseteq M\rtimes_\beta G_2$. Then a $w^*$-continuous surjective isometric $M$-bimodule map
$\psi:X_1\to X_2$ extends to a $*$-isomorphism $\phi:W^*(X_1)\to W^*(X_2)$.
This is obtained by letting $G:=G_1\times G_2$ act on $M\oplus M$ by the product action $\gamma:=\alpha\oplus \beta$ and applying Theorem \ref{liftX0} to
$X:=X_1\oplus X_2$ and the map $\theta(x_1\oplus x_2)=\psi^{-1}(x_2)\oplus \psi(x_1)$. The resulting $*$-automorphism of $W^*(X_1)\oplus W^*(X_2)$, when restricted to
$W^*(X_1)$, gives the desired $*$-isomorphism $\phi$.

\noindent (ii) \quad If the $M$-bimodule $X$ happens to be a subalgebra $A\subseteq M$ (possibly non-self-adjoint), then it follows from Theorem \ref{liftX0} that the $M$-bimodule maps of this theorem are automatically algebraic isomorphisms of $A$. \hfill$\square$
\end{rem}

\end{document}